\documentclass[a4paper,11pt]{amsart}
\usepackage{amsthm,amssymb,latexsym, amsmath}
\input amssym.def

\title[Logarithmic mean oscillation on the polydisc and paraproducts]{Logarithmic mean oscillation on the polydisc, endpoint results for
multi-parameter paraproducts, and commutators on $\BMO$}
\newtheorem{theorem}{T{\hskip 0pt\footnotesize\bf HEOREM}}[section]
\newtheorem{lemma}[theorem]{L{\hskip 0pt\footnotesize\bf EMMA}}
\newtheorem{proposition}[theorem]{P{\hskip 0pt\footnotesize\bf ROPOSITION}}
\newtheorem{definition}[theorem]{D{\hskip 0pt\footnotesize\bf EFINITION}}
\newtheorem{corollary}[theorem]{C{\hskip 0pt\footnotesize\bf OROLLARY}}

\newtheorem{question}[theorem]{Q{\hskip 0pt\footnotesize\bf UESTION}}

\def\eins{\mathbf{1}}
\def\O{\Omega}

\def\D{\mathcal D}
\def\RR{\mathcal R}
\def\J{\mathcal J}

\def\ZZ{\mathbb{Z}}

\def\vda{\vec{\delta}}
\def\Da{\Delta}

\def\veps{\vec{\varepsilon}}

\def\vba{\vec{\beta}}

\def\vj{\vec{j}}
\def\vk{\vec{k}}
\newcommand{\supp}{\mathrm{supp}}

\def\BMO{\mathrm{BMO}}
\def\LMO{\mathrm{LMO}}

\newcommand{\bprop} {\begin{proposition}}
\newcommand{\eprop} {\end{proposition}}
\newcommand{\btheo} {\begin{theorem}}
\newcommand{\etheo} {\end{theorem}}
\newcommand{\blem} {\begin{lemma}}
\newcommand{\elem} {\end{lemma}}
\newcommand{\bcor} {\begin{corollary}}
\newcommand{\ecor} {\end{corollary}}

\newcommand{\Be}{\begin{equation}}
\newcommand{\Ee}{\end{equation}}
\newcommand{\Bea}{\begin{eqnarray}}
\newcommand{\Eea}{\end{eqnarray}}
\newcommand{\Bes}{\begin{equation*}}
\newcommand{\Ees}{\end{equation*}}
\newcommand{\Beas}{\begin{eqnarray*}}
\newcommand{\Eeas}{\end{eqnarray*}}
\newcommand{\Ba}{\begin{array}}
\newcommand{\Ea}{\end{array}}
\def\R{\mathbb{R}}

\def\N{\mathbb{N}}

\def\oN{\overline{N}}
\def\uN{\underline{N}}
\def\oK{\overline{K}}
\def\uK{\underline{K}}

\def\T{\mathbb{T}}
 \scrollmode

\begin{document}

\author{Sandra Pott
and Benoit Sehba}
\address{Sandra Pott,
 Centre for Mathematical Sciences, Faculty of Science, Lund University, 22100 Lund, Sweden}
\email{sandra.pott@math.lu.se}
\address{Beno\^it Sehba, School of Mathematics, Trinity College, Dublin 2, Ireland.}
\email{sehbab@tcd.ie}

\keywords{Paraproduct, Haar basis, bounded mean oscillation, logarithmic mean oscillation.}
\subjclass[2000]{Primary: 42B30, 42B37, Secondary: 42B20 }

\thanks{The first author was supported by a Heisenberg fellowship of the German Research Foundation (DFG), the second
author acknowledges funding from the ``Irish
Research Council for Science, Engineering and Technology". }
\begin{abstract}
We study boundedness properties of a class of  multiparameter paraproducts on the dual space of the dyadic Hardy space
$H^1_d(\T^N)$, the dyadic product $\BMO$ space $\BMO^d(\T^N)$. For this, we introduce a notion of logarithmic mean oscillation on the polydisc.
We also obtain a result on the boundedness of iterated commutators on $\BMO([0,1]^2)$.

\end{abstract}
\maketitle
\section{Introduction and notation}
In recent years, multi-parameter paraproducts have generated much interest \cite{bp,ferglac, laceyjason, muscalu},
both in their own right
and as building blocks for other operators, such as commutators and Hankel operators.

In this paper, we characterize boundedness of dyadic paraproducts on the
endpoint spaces  $\BMO^d(\T^N)$ and $H^1_d(\T^N)$.
Here, the spaces $H^1(\R^N)$ and $\BMO(\R^N)$ and their dyadic counterparts $\BMO^d(\T^N)$ and $H^1_d(\T^N)$ on the polydisc
are the product spaces in the sense of Chang and Fefferman \cite{ChFef2}.

Our main interest will be for the paraproduct denoted below
by $\Pi$ on the space $\BMO^d(\T^N)$. We will  prove a characterization of boundedness in terms of a natural notion of logarithmic mean oscillation in the polydisc.

We then apply the results on paraproducts to obtain a result on the boundedness of iterated commutators with the Hilbert transforms on compactly supported functions in
$\BMO(\R^2)$. This is motivated by the classical  one-parameter
results in \cite{JPS} on Hankel operators, or equivalently commutators with the Hilbert transform, on $\BMO(\T)$, and by the more recent results of Ferguson, Lacey and
Terwilleger on iterated commutators with the Hilbert transforms on $L^2(\R^N)$, see  \cite{ferglac, laceyterw}.

The notion of logarithmic mean oscillation was originally introduced in the one-parameter setting
for the characterization of multipliers of $\BMO$ and Toeplitz operators on $H^1$ \cite{stegenga, zhao}. The corresponding multiparameter results, which rely on our results here,
 are the subject of a forthcoming paper
\cite{ pottsehba2}.

The paraproduct denoted by $\Delta$ below and continuous analogues have been considered on $\BMO^d(\T^N)$ and $\BMO(\T^N)$ before,
see \cite{bp, laceyjason}. We restrict most of our presentation to the two-dimensional case. As the general case  follows in
the same way, we will just give the corresponding results.

\vspace{0.5cm}
The paper is organised as follows. In Section \ref{sec:main}, we prove the main technical results on the paraproduct $\Pi$. In Section \ref{sec:others}, we give conditions on the boundedness of
the other paraproducts. The general $N$-parameter case is treated in Section \ref{sec:multi}. In Section \ref{sec:hankel}, we first consider paraproducts of functions on $\R^N$
rather than $\T^N$. For this, local versions
of the results of Sections \ref{sec:main} and \ref{sec:others} are required. These results are then used to prove boundedness estimates for commutators with the so-called dyadic shift on product
$\BMO^d$. These  in turn lead to a result on the boundedness
of iterated commutators with the Hilbert transforms on a suitable product $\BMO$ space,
by means of
 the decomposition of the Hilbert transform into dyadic shifts and new results on the relation between dyadic and continuous product $\BMO$ spaces.

 \subsection*{Notation}
Let $\mathbb {T}$ denote the unit circle. We identify $\T$ with the interval $[0,1)$ in the usual way and write $\mathcal D$
for the set of all dyadic subintervals. We denote by
$\mathcal R$ the set of all dyadic rectangles $R=I\times J$, where $I$
and $J$ in $\mathcal D$. Let $h_I$ denote the Haar wavelet adapted
to the dyadic interval $I$,
$$
h_I=|I|^{-1/2}(\chi_{I^+}-\chi_{I^-}),
$$
where $I^+$ and $I^-$
are the right and left halves of $I$, respectively.

For any rectangle $R\in \mathcal {R}$, the product Haar wavelet
adapted to $R=I\times J = h_I \otimes h_J$ is defined by $h_R(s,t)=h_I(s)h_J(t)$.
These wavelets form an orthonormal basis of
$$
L_0^2(\mathbb
{T}^2) = \left\{ f \in L^2(\T^2):  \int_\T f(s,t) dt =0, \int_\T f(s, t) ds =0 \text{ for a.e. }s,t \in \T \right\},
$$
with
$$f=\sum_{R\in \mathcal {R}} \langle f,h_R \rangle h_R=\sum_{R\in \mathcal
{R}}f_Rh_R \qquad (f \in L^2_0(\T^2)).
$$
We will be writing $m_Rf$ for the mean of $f\in
L^2(\mathbb {T}^2)$ over the dyadic rectangle $R= I \times J$ and $f_R= f_{IJ}$ for the Haar coefficient $\langle f, h_R \rangle = \langle f, h_I \otimes h_J \rangle   $.

The space of functions of dyadic bounded mean oscillations in $\mathbb T^2$,
$\BMO^d(\mathbb T^2)$, is the space of all function $b\in L_0^2(\mathbb
T^2)$ such that
\begin{equation}
 ||b||_{\BMO^d}^2:=\sup_{\O\subset \mathbb
{T}^2}\frac{1}{|\O|}\sum_{R\in \O}|b_R|^2=\sup_{\O\subset \mathbb
{T}^2}\frac{1}{|\O|}||P_\O b||_2^2 < \infty,
\end{equation}
where the supremum
is taken over all open sets $\O\subset \mathbb T^2$ and $P_\O$ is the
orthogonal projection on the subspace spanned by Haar functions
$h_R$, $R\in \mathcal R$ and $R\subset  \O$.

It is well-known (see e.g.
\cite{ChFef1}, \cite{bp}) that $\BMO^d(\mathbb T^2)$ is the dual space of the
dyadic product Hardy space $H^{1}_d(\mathbb T^2)$ defined in terms of the
dyadic square functions $S$.

That means,
$$
H_d^{1}(\mathbb T^2)=\{f\in L_0^1(\mathbb T^2): S[f]\in
L^1(\mathbb T^2)\},
$$
where
\begin{equation}   \label{eq:sq}
     S[f] = \left(\sum_{R \in \RR}  \frac{\chi_R}{|R|} |f_R|^2 \right)^{1/2}.
\end{equation}

For $I$ a dyadic interval and $\varepsilon\in \{0,1\}$, we define $h^{\varepsilon}_I$ by
$$h^{\varepsilon}_I =\left\{ \begin{matrix} h_I &\text{if }& \varepsilon=0\\
      |I|^{-1/2}|h_I| & \text{ if } & \varepsilon=1
                                  \end{matrix} \right.
$$

For $R=I\times J\in \RR$ and $\vec {\varepsilon}=(\varepsilon_1,\varepsilon_2)$, with $\varepsilon_j\in \{0,1\}$, we write
$$
h^{\vec {\varepsilon}}_R=h^{\varepsilon_1}_I \otimes h^{\varepsilon_2}_J.
$$
We will consider operators of the following general form:
\begin{equation}\label{paraprodgene2}B_{\veps,\vda,\vba}(\phi,f):=\sum_{R\in \RR}\langle \phi,h^{\vec {\varepsilon}}_R\rangle \langle f,h^{\vec {\delta}}_R\rangle h^{\vec {\beta}}_R.\end{equation}

They appear naturally in the study of many other operators in complex analysis and harmonic analysis. In this note, we consider the paraproducts
appearing as pieces of the usual product in the Haar expansion and corresponding to non-diagonal terms in this expansion. Some of the other operators of the form given in
(\ref{paraprodgene2})
on endpoint spaces appear in \cite{pottsehba2}.

 In other words, we consider here paraproducts
$B_{\veps,\vda,\vba}(\phi, \cdot)$ with symbol $\phi$ corresponding to triples $(\veps, \vda, \vba)$ with $\vec {\varepsilon}=(0,0)$ and
$$
\delta_j=\left\{ \begin{matrix} 1 &\text{if}& \beta_j=0\\

                0 &\text{otherwise}&.
                              \end{matrix}\right.
$$
Finally, for simplicity, we can just denote the corresponding paraproducts by $\Pi^{\vec {\beta}}$. One easily sees that there are exactly four in dimension $N=2$. We will occasionally use the notation
$\vec{1}=(1,1)$, $\vec{0} = (0,0)$.

\vspace{0.5cm}

As usual, for $\vj =(j_1, j_2) \in \N_0 \times \N_0$ we define the $j_1$th generation of dyadic intervals and the $\vj$th generation of dyadic rectangles,
$$
     \D_{j_1} = \{I \in \D: |I|= 2^{-j_1} \},
$$
$$
    \RR_{\vj} = \D_{j_1} \times \D_{j_2} = \{ I \times J \in \RR: |I|= 2^{-j_1}, |J|=2^{-j_2}\},
$$
%
the product Haar martingale difference,
$$
   \Da_{\vj} f = \sum_{R\in \RR_{\vj}} \langle  f, h_R \rangle h_R,
$$
the expectations
$$
   E_{\vj} f = \sum_{\vk \in \N_0 \times \N_0,  \vk < \vj } \Da_{\vk}f,
$$
where we write $(k_1, k_2) = \vk < \vj = (j_1, j_2)$ for $k_1 < j_1$, $k_2 < j_2$ and correspondingly $(k_1, k_2) = \vk \le \vj = (j_1, j_2)$ for $k_1 \le j_1$, $k_2 \le j_2$,
$$
   E^{(1)}_{i} f = \sum_{\vk  \in \N_0 \times \N_0,  k_1 < i} \Da_{\vk}f,
$$
and
$$
   E^{(2)}_{j} f = \sum_{ \vk  \in \N_0 \times \N_0,    k_2 < j} \Da_{\vk}f,
$$
for $f \in L^2(\T)$, $\vj \in \N_0 \times \N_0$, $i,j \in \N_0$.

We will also require the operators on $L^2(\T^2)$ given by
\begin{equation}   \label{eq:qdef}
   Q_{\vj}f =\sum_{ \vk \ge \vj } \Da_{\vk}f
\end{equation}
$$
   Q^{(1)}_{i}f =\sum_{ \vk  \in \N_0 \times \N_0, k_1 \ge i} \Da_{\vk}f
$$
$$
   Q^{(2)}_{j}f =\sum_{ \vk  \in \N_0 \times \N_0, k_2 \ge j} \Da_{\vk}f.
$$
Note that contrary to the one-parameter situation, $Q_{\vk}$ is not the orthogonal complement of the expectation $E_{\vk}$. In fact, we have the relation
\begin{equation}    \label{eq:EQ}
        f = E_{\vk} f +  E^{(1)}_{k_1} Q^{(2)}_{k_2} f +Q^{(1)}_{k_1} E^{(2)}_{k_2} f  + Q_{\vk} f \text{ for } \vk \in \N_0 \times \N_0, \; f \in L^2(\T^2).
\end{equation}
Let $\phi \in L^2(\T^2)$. The (main) paraproduct $\Pi_\phi$ is defined by
$$
   \Pi_\phi f = \Pi(\phi,f):= \sum_{\vj \in \N_0 \times \N_0}
(\Delta_{\vj} \phi) (E_{\vj} f) = \sum_{R\in \RR} h_R \phi_R m_R f
$$
on functions with finite Haar expansion. This is just the paraproduct $\Pi^{(0,0)}$ introduced above.

We will now define the space of functions of dyadic logarithmic mean oscillation on the bidisc, $\LMO^d(\T^2)$.
\begin{definition}   \label{def:lmod}
Let $\phi \in L^2(\T^2)$. We say that $ \phi \in \LMO^d(\T^2)$, if there exists $C >0$ with
$$
    \|Q_{\vj} \phi\|_{\BMO^d(\T^2)} \le C \frac{1}{(j_1 +1)(j_2+1)}
$$
for all $\vj = (j_1, j_2)  \in \N_0 \times \N_0$. The infinimum of such constants is denoted by $\|\phi\|_{\LMO^d}$.
\end{definition}

An alternative characterization, which is closer in spirit to the one-parameter case, is the following:
\begin{proposition} \label{prop:char}
Let $\phi \in L^2(\T^2)$. Then $ \phi \in \LMO^d(\T^2)$, if and only if there exists $C >0$ such that for each
dyadic rectangle $R= I \times J$ and each open set $\Omega \subseteq R$,
\begin{equation}   \label{eq:lmochar}
   \frac{(\log\frac{4}{|I|})^2 (\log\frac{4}{|J|})^2}{|\Omega|} \sum_{Q \in \RR, Q \subseteq \Omega} |\phi_Q|^2 \le C.
\end{equation}
\end{proposition}

\proof Let $\phi \in \LMO^d(\T^2)$ in the sense of Definition \ref{def:lmod},  let $R= I \times J$ be a dyadic rectangle with $|I| = 2^{-j_1}$, $|J| = 2^{-j_2}$,
 and let $\Omega \subseteq R$ be open. Let $\vj = (j_1, j_2)$. Then
 \begin{multline*}
      \sum_{Q \in \RR, Q \subseteq \Omega} |\phi_Q|^2 = \|P_\Omega \phi \|_2^2 = \|P_\Omega Q_{\vj} \phi \|_2^2 \le |\Omega| \| Q_{\vj} \phi \|^2_{\BMO^d}    \\
      \le
       |\Omega| \frac{1}{(j_1 +1)^2 (j_2 +1)^2}    \| \phi \|^2_{\LMO^d} \lesssim  (\log\frac{4}{|I|})^{-2} (\log\frac{4}{|J|})^{-2} |\O|  \| \phi \|^2_{\LMO^d}.
 \end{multline*}

 Conversely, suppose that $\phi \in L^2(\T^2)$ and that (\ref{eq:lmochar}) holds. Let $\vj = (j_1, j_2) \in \N_0^2$, and let $\Omega \subseteq \T^2$ open.
Then
\begin{multline*}
   \| P_\Omega Q_{\vj} \phi \|_2^2 = \sum_{R \in \RR_j}        \| P_{R \cap \Omega} Q_{\vj} \phi \|_2^2 \\
       \lesssim C \frac{1}{(j_1+1)^2} \frac{1}{(j_2+1)^2}  \sum_{R \in \RR_j}  |R \cap \Omega|  = C \frac{1}{(j_1 +1)^2 (j_2 +1)^2} |\O|.
\end{multline*}
This holds for all $\Omega \subseteq \T^2$ open, hence $\|Q_{\vj} \phi \|_{\BMO^d} \lesssim \frac{1}{(j_1 +1)(j_2 +1)}$.
\qed

\section{The main paraproduct}     \label{sec:main}
Here is our main result of this section.
\begin{theorem}   \label{thm:main}
Let $\phi \in L^2(\T^2)$. Then $\phi \in \LMO^d(\T^2)$, if and only if $\Pi_\phi:\BMO^d(\T^2) \rightarrow \BMO^d(\T^2)$ is bounded, and
$\|\Pi_\phi\|_{\BMO \to \BMO} \approx \|\phi\|_{\LMO^d}$.
\end{theorem}

Let us introduce some more notations. Given an integrable function $f$ on $\T^2$ and intervals $I$ and $J$ in $\T$. We write
$m_If=\frac{1}{|I|}\int_If(s,t)ds$, $m_Jf=\frac{1}{|J|}\int_Jf(s,t)dt$ and, $m_Rf=\frac{1}{|R|}\int_Rf(s,t)dsdt$,
$R=I\times J$. We remark that $m_If$ is in fact a function of the second variable while $m_Jf$ is a function in the first
variable. We will require the following lemma on the growth of averages and restrictions of functions in $\BMO$.
\begin{lemma}  \label{lemma:avgrowth}  Let $b \in \BMO^d(\T^2)$, $\vk = (k_1, k_2) \in \N_0 \times \N_0$. Then
$$
  |m_R b| \lesssim (k_1+1)(k_2 +1) \|b\|_{\BMO^d(\T^2)} \quad  (R \in \RR_{\vk});
$$
$$
  \|m_I b\|_{\BMO^d(\T)} \lesssim (k_1+1) \|b\|_{\BMO^d(\T^2)} \quad  (I \in \D_{k_1});
$$
$$
    \|\chi_R b \|_2^2 \lesssim (k_1+1)^2(k_2+1)^2 |R| \|b\|^2_{\BMO^d(\T^2)} \quad   (R \in \RR_{\vk});
$$
$$
    \|\chi_I P_J b \|_2^2 \lesssim (k_1+1)^2 |I| |J|  \|b\|^2_{\BMO^d(\T^2)} \quad   (R \in \RR_{\vk});
$$
and this is sharp.
\end{lemma}
\begin{proof} Let $R = I \times J$, $|I|=2^{-k_1}$, $J= 2^{-k_2}$, $\vk = (k_1, k_2)$, $j \in \N_0$. For the first inequality, consider
\begin{eqnarray*}
   \sup_{b \in \BMO^d, \|b\|_{\BMO^d} =1} |m_R b|
 & =& \sup_{b \in \BMO^d, \|b\|_{\BMO^d} =1} | \langle b, \frac{\chi_R}{|R|} \rangle| \\
&\lesssim& \|\frac{\chi_R}{|R|}\|_{H^1_d(\T^2)} \\
&=& \|\frac{\chi_I}{|J|}\|_{H^1_d(\T)} \|\frac{\chi_J}{|J|}\|_{H^1_d(\T)} \\
&\lesssim& \log(\frac{4}{|I|})  \log(\frac{4}{|J|}) \approx (k_1+1)(k_2+1),
\end{eqnarray*}
where we use the $H^1_d(\T^2)$- $\BMO^d(\T^2)$ duality in the first line
and the known one-variable results in the last line.

For the second inequality, note that
\begin{eqnarray*}
\|m_I b\|_{\BMO^d(\T)} &\approx & \sup_{f \in H^1_d(\T), \|f\|_{H^1_d} \le 1}
 |\int_\T \int_\T \frac{1}{|I|}\chi_{I}(s) f(t) b(s,t) ds dt| \\
&\lesssim &
\|b\|_{\BMO^d(\T^2)} \,\sup_{f \in H^1_d(\T), \|f\|_{H^1_d} \le 1} \|\frac{1}{|I|}\chi_I(s)f(t)\|_{H^1_d(\T^2)} \\
&=&  \|b\|_{\BMO^d(\T^2)}  \|\frac{1}{|I|}\chi_I(s)\|_{H^1_d(\T)} \lesssim  (k_1+1) \|b\|_{\BMO^d(\T^2)}.
\end{eqnarray*}

For the third inequality, write $\chi_R b(s,t) = P_R b(s,t) +
  \chi_R(s,t) m_I b(t) + \chi_R(s,t) m_J b(s) - \chi_R(s,t) m_R b$.

Clearly $\|P_R b \|_2^2 \le |R| \|b\|^2_{\BMO^d}$ and
$$\|\chi_R m_Rb \|_2^2 = |m_R
b|^2 |R| \lesssim (k_1+1)^2 (k_2+1)^2 |R| \|b\|^2_{\BMO^d}$$ by the first inequality in Lemma
\ref{lemma:avgrowth}. The results for the remaining terms follow from
the
one-dimensional John-Nirenberg inequality, since e.~g.
\begin{eqnarray*}
\|\chi_R(s,t) m_I b(t)\|_2 &=& |I|^{1/2} \| \chi_J(t) m_I b(t) \|_2 \\
&\lesssim& |I|^{1/2}  |J|^{1/2}(k_2+1)  \| m_I b(t) \|_{\BMO^d(\T)} \\
&\lesssim& |I|^{1/2}  |J|^{1/2}  (k_1+1) (k_2+1)  \| b \|_{\BMO^d(\T^2)} \\
\end{eqnarray*}
by the second inequality.

For the last inequality, note that
\begin{eqnarray*}
\|\chi_I(s) m_I (P_J b)\|^2_2 &=& |I|  \| m_I(P_J b) \|_2^2 \\
&\lesssim& |I|  |J|  \| m_I b(t) \|_{\BMO^d(\T)} \\
&\lesssim& |I|  |J| (k_1+1)^2  \| b \|_{\BMO^d(\T^2)} \\
\end{eqnarray*}
by the second inequality.

Hence
\begin{eqnarray*}
\|\chi_I(s) P_J b  \|_2 &\le& \|\chi_I(s) m_I (P_J b)\|_2 + \|P_{I \times J} b\|_2 \\
&\le& |I|^{1/2}  |J|^{1/2} (k_1+1)  \| b \|_{\BMO^d(\T^2)} + |I|^{1/2}  |J|^{1/2}  \| b \|_{\BMO^d(\T^2)} \\
& \lesssim &  |I|^{1/2}  |J|^{1/2} (k_1+1)  \| b \|_{\BMO^d(\T^2)}
\end{eqnarray*}
by the second inequality.

Sharpness follows easily from the one-dimensional case, forming an
appropriate product of $\BMO^d(\T)$ functions in the two different variables.
\end{proof}

Next, for $\vk \in \N_0 \times \N_0$ and $b\in L^2(\T^2)$, we consider the operator
$\Pi_bE_{\vk} = \Pi(b, E_{\vk}\, \cdot)$ on $L^2(\T^2)$,  given by
$$\Pi_bE_{\vk}f=\Pi(b,E_{\vk}f), \,\,\,f\in L^2(\T^2).
$$
\begin{lemma} \label{lemma:core}  Let $b \in L^2(\T^2)$ and let $\vk =(k_1, k_2) \in \N_0 \times \N_0$. Then
$$
   \|\Pi_b E_{\vk} \|_{L^2 \to L^2} = \| \Pi({\sigma_{\vk}(b)}, \cdot )\|_{L^2 \to L^2},
   $$
where $\sigma_k b$ is given by
$$
   (\sigma_{\vk}b)_{I,J} =\left\{ \begin{matrix} b_{I,J} & \text{ if } & |I|
   > 2^{-k_1}, |J|>2^{-k_2} \\
      (\sum_{J' \subseteq J} |b_{I,J'}|^2)^{1/2}   & \text{ if } &
   |I| > 2^{-k_1}, |J|=2^{-k_2}\\
 (\sum_{I' \subseteq I} |b_{I',J}|^2)^{1/2}   & \text{ if } &
   |I| = 2^{-k_1}, |J|>2^{-k_2}\\
     (\sum_{I' \subseteq I, J' \subseteq J} |b_{I',J'}|^2)^{1/2}   &
   \text{ if } &  |I| = 2^{-k_1}, |J|=2^{-k_2}\\
      0 & \text{ otherwise. }&
                               \end{matrix} \right.
$$
In particular,
$$
   S^2[\sigma_{\vk}b]= E_{\vk} S^2[b],
$$
where $S^2[b]$ denotes the square of the dyadic square function in (\ref{eq:sq}),
and
$$
         \| \sigma_{\vk} b\|_2 = \|b \|_2.
$$

\end{lemma}
\begin{proof}
Let $f \in L^2(\T^2)$. Then
\begin{eqnarray*}
\|\Pi_b E_{\vk} f \|^2 &=& \|\sum_{\vj} (\Delta_{\vj} b) E_{\vj} E_{\vk} f \|^2 \\
  &=& \sum_{\vj \ge \vk}  \|(\Delta_{\vj} b) E_{\vk} f \|^2 +
     \sum_{j_1\ge k_1, j_2 < k_2}  \|(\Delta_{\vj} b)  E_{(k_1, j_2)} f \|^2 \\
 && \quad +
     \sum_{j_1< k_1, j_2 \ge k_2}  \|(\Delta_{\vj} b) E_{(j_1, k_2)} f \|^2 +
      \sum_{\vj < \vk}  \|(\Delta_{\vj} b) E_{\vj}  f \|^2 \\
&=&  \|(\sum_{\vj \ge \vk}  |\Delta_{\vj} b|^2 )^{1/2} E_{\vk}f \|^2 +
      \sum_{j_2 < k_2 }  \|(\sum_{j_1 \ge k_1}|\Delta_{\vj} b|^2)^{1/2} E_{(k_1, j_2)} f \|^2 \\
 &&  \quad +
       \sum_{j_1 < k_1}  \|(\sum_{j_2 \ge k_2}|\Delta_{\vj} b|^2)^{1/2} E_{(j_1,k_2)} f \|^2 +
      \sum_{\vj < \vk }  \|(\Delta_{\vj} b) E_{\vj} f \|^2 \\
 &=& \|\sum_{\vj \le \vk} \Delta_{\vj} (\sigma_{\vk} b) E_{\vj}  f \|^2
= \|\Pi(\sigma_{\vk} b, f) \|^2.
\end{eqnarray*}

The remaining identities for $\sigma_{\vk}b$ follow directly from the definition.
\end{proof}

Here is our main technical lemma.
\begin{lemma} \label{lemma:core2}  Let $\phi, b \in \BMO^d(\T^2)$ and $\vk = (k_1, k_2) \in \N_0 \times \N_0$.
  Then
$$  \|\Pi\left(\Pi(\phi,b), E_{\vk}\,\,  \cdot \right)\|_{L^2 \to L^2} \lesssim (k_1+1)(k_2+1) \, \|\phi\|_{\BMO^d} \|b \|_{\BMO^d}.
           $$
\end{lemma}
\begin{proof}
By Lemma \ref{lemma:core}, we have to estimate the $\BMO^d$ norm of $\sigma_{\vk} (\Pi_{\phi} b)=\sigma_{\vk} (\Pi(\phi, b))$. Clearly
\begin{multline*}
   \sigma_{\vk} ( \Pi_\phi b) =  \sigma_{\vk} (E_{\vk} \Pi_\phi b) +
    \sigma_{\vk} (E_{k_1}^{(1)} Q_{k_2}^{(2)}\pi_\phi b) +  \sigma_{\vk}(E_{k_2}^{(2)}
   Q_{k_1}^{(1)}\Pi_\phi b) +  \sigma_{\vk}(Q_{\vk}\Pi_\phi b)\\
 = E_{\vk} \Pi_\phi b +
   \sigma_{\vk}(E_{k_1}^{(1)}Q_{k_2}^{(2)}\Pi_{ \phi} b) +  \sigma_{\vk}(E_{k_2}^{(2)}Q_{k_1}^{(1)}
   \Pi_{\phi} b) +  \sigma_{\vk}(Q_{\vk} \Pi_\phi b)  = I + II + III + IV.
\end{multline*}
(compare this decomposition to the one in the definition of  $\sigma_{\vk}$
     in Lemma \ref{lemma:core}).

We start with term I.
For any open set $\Omega \subseteq \T^2$,
\begin{eqnarray*}
    \frac{1}{|\Omega|} \|P_{\Omega} E_{\vk} \Pi_\phi b\|_2^2
&=& \frac{1}{|\Omega|} \sum_{R = I \times J, |I| > 2^{-k_1}, |J|> 2^{-k_2},
   R \subset \Omega}  |\phi_R|^2 |m_R b|^2 \\
& \lesssim & \frac{(k_1+1)^2 (k_2+1)^2}{|\Omega|} \sum_{R = I \times J, |I| >
2^{-k_1}, |J|> 2^{-k_2},
   R \subset \Omega}  |\phi_R|^2   \|b\|^2_{\BMO^d} \\
&\lesssim & (k_1+1)^2 (k_2+1)^2 \| \phi \|^2_{\BMO^d}  \|b\|^2_{\BMO^d}.
\end{eqnarray*}
by Lemma \ref{lemma:avgrowth}.

For term II, note that
since $\sigma_{\vk}({E^{(1)}_{k_1}
    Q^{(2)}_{k_2} \Pi_\phi b})$ has only nontrivial Haar coefficients for
those $R = I \times J$ with $|J| = 2^{-k_2}$ and $|I| > 2^{-k_1}$ (this corresponds to the second term in the definition of $\sigma_{\vk}
$ in Lemma \ref{lemma:core}), it is
sufficient to check the BMO norm on rectangles $R = I \times J$ with
$|J| = 2^{-k_2}$ and $|I| > 2^{-k_1}$. Then
\begin{eqnarray*}
   \frac{1}{|R|}\|P_R  \sigma_{\vk}({E^{(1)}_{k_1}
    Q^{(2)}_{k_2} \Pi_\phi b})\|^2_2  &=& \frac{1}{|R|}\sum_{I' \subseteq I} \left|\left(\sigma_{\vk}({E^{(1)}_{k_1}
    Q^{(2)}_{k_2} \Pi_\phi b})\right)_{I',J}\right|^2 \\
     &=& \frac{1}{|R|}\sum_{I' \subseteq I, J' \subseteq J, I'\times J' \in \RR }
   |\phi_{I' \times J'}|^2 |m_{I' \times J'} b|^2 \\
&\le& \frac{1}{|R|}\| \Pi_\phi \chi_{R} b \|_2^2 \lesssim
\|\phi\|_{\BMO^d}^2 \frac{1}{|R|}\|
    \chi_R b\|^2_2\\ &\lesssim& (k_1+1)^2 (k_2+1)^2  \|\phi\|_{\BMO^d}^2 \|b\|_{\BMO^d}^2.
\end{eqnarray*}
Term III is dealt with analogously.
For term IV,
note that since $\sigma_{\vk}( Q_{\vk} \Pi_\phi b)$
 has only nontrivial Haar
coefficient for $R \in \RR_{\vk}$, it is enough to check the BMO norm
on rectangles of this type, and we obtain for $R= I \times J \in \RR_{\vk}$:
\begin{eqnarray*}
 \frac{1}{|R|} \int_R |P_R \sigma_{\vk}(Q_{\vk} \Pi_\phi b )   |^2 ds dt
&\le& \frac{1}{|I||J|} \sum_{I' \subseteq I, J' \subseteq J}
  |\phi_{I', J'}|^2 |m_{I',J'} b|^2 \\
&=&  \frac{1}{|R|} \|\Pi_\phi \chi_R b\|^2_2 \\
& \lesssim &  \frac{1}{|R|} \|\phi\|^2_{\BMO^d} \| \chi_R b\|_2^2 \\
&\lesssim &   (k_1+1)^2(k_2+1)^2 \| \phi\|^2_{\BMO^d} \|b\|^2_{\BMO^d}
\end{eqnarray*}
by Lemma \ref{lemma:avgrowth}.
\end{proof}

In particular, we have
\begin{lemma} \label{lemma:core2bis}  Let $\phi \in \LMO^d(\T^2)$,  $b \in \BMO^d(\T^2)$ and $\vk,\vj \in \N_0 \times \N_0$.
  Then
$$  \|\Pi\left(\Pi(Q_{\vj}\phi,b), E_{\vk}\,\,  \cdot \right)\|_{L^2 \to L^2} \lesssim \frac{|\vk + \vec{1}|}{|\vec{j} + \vec{1}|} \, \|\phi\|_{\LMO^d} \|b \|_{\BMO^d},
           $$ where $|\vk + \vec{1}|=(k_1+1)(k_2+1)$ and $|\vec{j} + \vec{1}|=(j_1+1)(j_2+1)$.
\end{lemma}
\begin{proof} Definition \ref{def:lmod} and Lemma \ref{lemma:core2}.
\end{proof}

\begin{proof} of Theorem \ref{thm:main}. We begin by proving necessity. Suppose that \\
$\Pi_\phi: \BMO^d(\T^2) \rightarrow \BMO^d(\T^2)$ is bounded.  Let $R = I \times J$ be a dyadic rectangle, with $|I| =2^{-k}$
and $|J|= 2^{-l}$, and let
$\Omega \subseteq R$ be open. It is easy to see that there exists a function $b \in \BMO^d(\T^2)$ with
$$
       b|_R \equiv (k+1)(l+1) \text{ and }  \|b\|_{\BMO^d} \le C,
$$
 where $C$ is a constant independent of $R$.
Such a function can for example be found by forming the product $b_1 \otimes b_2$ of two one-variable functions $b_1$, $b_2$,
 which have the corresponding properties for the intervals $I$ and $J$, respectively.
For details on
the construction in the one-dimensional case, see e.~g.~\cite{benoit}.
Then
\begin{multline*}
    \frac{(\log\frac{4}{|I|})^2 (\log\frac{4}{|J|})^2}{|\Omega|} \sum_{Q \in \RR, Q \subseteq \Omega} |\phi_Q|^2
 \approx  \frac{(k+1)^2(l+1)^2}{|\Omega|} \sum_{Q \in \RR, Q \subseteq \Omega} |\phi_Q|^2 \\
   =  \frac{1}{|\Omega|} \sum_{Q \in \RR, Q \subseteq \Omega} |\phi_Q|^2 |m_Qb|^2
  \le \|\Pi_\phi b\|^2_{\BMO^d} \le C^2 \|\Pi_\phi\|^2_{\BMO^d \to \BMO^d}.
\end{multline*}
Thus $\phi \in \LMO^d(\T^2)$ by Proposition \ref{prop:char}, with the appropriate norm estimate.

To prove sufficiency of the $\LMO^d$ condition for boundedness of the paraproduct on $\BMO^d$,
let $\phi \in \LMO^d(\T^2)$ and $b \in \BMO^d(\T^2)$. We will estimate
$ \| \Pi_\phi b \|_{\BMO^d} \approx \| \Pi_{\Pi(\phi,b)}\|_{L^2 \to L^2}=\| \Pi\left(\Pi(\phi,b),\cdot\right)\|_{L^2 \to L^2}$ by mean of Cotlar's Lemma,
and use Lemma \ref{lemma:core2bis} to control off-diagonal decay.

For $N,K \in \N_0$, let
\begin{equation} \label{eq:pnk}
\begin{split}
   P_{N,K} &= \sum_{j_1=2^N-1}^{2^{N+1}-2} \sum_{j_2=2^K-1}^{2^{K+1}-2}
   \Da_{\vj}, \\
P^{N,K} &= \sum_{j_1=2^N-1}^\infty \sum_{j_2=2^K-1}^\infty \Da_{\vj},
\end{split}
\end{equation}
and
$$
   T_{N,K} =  \Pi_{\Pi(\phi, b)} P_{N,K}=\Pi\left(\Pi(\phi,b), P_{N,K}\, \cdot\right).
$$

That means, we wish to estimate the $L^2 - L^2$ operator norm of $ \Pi(\Pi(\phi, b), \cdot ) = \sum_{N,K=0}^\infty T_{N,K}$.
Clearly $T_{N,K} T_{N',K'}^* =0$ for $N \neq N'$ or $K \neq K'$. Therefore, we only have to estimate
the norm of $T_{N,K}^* T_{N',K'}$ for $N, N', K,K' \in \N$. Letting
$\overline{N}= \max\{N,N'\}$, $\underline{N}= \min\{N,N'\}$,
    $\overline{K}= \max\{K,K'\}$, $\underline{K}= \min\{K,K'\}$
    and using the elementary identities
   \begin{multline*}
     \Pi(\Pi(\phi, b), \cdot) P_{N,K}  = \Pi(\Pi(P^{N,K} \phi, b), \cdot) P_{N,K}  \\ =  \Pi(P^{N,K} \Pi( \phi, b), \cdot) P_{N,K}=  P^{N,K} \Pi(\Pi( \phi, b), \cdot) P_{N,K}
   \end{multline*}
and
 $$
  P^{N,K}   \Pi(\Pi(\phi, b), \cdot)   = \Pi(\Pi(P^{N,K} \phi, b), \cdot),
 $$
we obtain
\begin{eqnarray*}
  \|T_{N,K}^* T_{N',K'}\| &=& \|P_{N,K} \left(\Pi(\Pi(\phi, b), \cdot )\right)^*\Pi(\Pi(\phi, b), P_{N',K'}\,\cdot)\| \\
&=&  \|P_{N,K} \Pi({ \Pi(\phi,b)}, \cdot) )^* P^{N,K} P^{N',K'} \Pi({\Pi(\phi,b)}, \cdot)
  P_{N',K'}\| \\
 &=&  \|P_{N,K} \Pi({ \Pi(\phi,b)}, \cdot)^* P^{\oN,\oK} \Pi({ \Pi(\phi,b)}, \cdot)
  P_{N',K'}\| \\
&=&  \|P_{N,K} \Pi( \Pi (P^{\oN,\oK} \phi ,b), \cdot )^*  \Pi( \Pi (P^{\oN,\oK} \phi ,b), \cdot )
  P_{N',K'}\| \\
 &\le &  \|P_{N,K} \Pi( \Pi (P^{\oN,\oK} \phi ,b), \cdot )^* \|  \| \Pi( \Pi (P^{\oN,\oK} \phi ,b), \cdot )
  P_{N',K'}\| \\
   &\le &  \| \Pi( \Pi (P^{\oN,\oK} \phi ,b), P_{N,K} \cdot ) \|  \| \Pi( \Pi (P^{\oN,\oK} \phi ,b),P_{N',K'}  \cdot )    \| \\
&\lesssim& \frac{{2^{\uN +1}}{2^{\uK +1}}}{{2^{\oN}}{2^{\oK}}}
                     \frac{{2^{\oN +1}}{2^{\oK +1}}}{{2^{\oN}}{2^{\oK}}}
      \|\phi\|^2_{\LMO^d}\|b\|^2_{\BMO^d}\\
&\lesssim & 2^{-|N-N'|} 2^{-|K-K'|} \|\phi\|^2_{\LMO^d}\|b\|^2_{\BMO^d}
\end{eqnarray*}
by Lemma \ref{lemma:core2bis}. In particular, the $T_{N,K}$ are uniformly bounded in norm,
and there exists a positive sequence $(\alpha(i,j))_{i, j \ge 0}$ with
$$
   \sum_{i,j=1}^\infty \alpha(i,j)^{1/2} < \infty
$$
such that
$$
   \|T^*_{N,K} T_{N',K'}\| \le \alpha(|N-N'|, |K-K|) \, \|\phi\|^2_{\LMO^d}\|b\|^2_{\BMO^d}.
$$
Thus, by Cotlar's Lemma, $T= \Pi(\Pi(\phi, b),\cdot)$ is bounded, and there exists an absolute constant $C>0$ with
$$
\|\Pi\left(\Pi(\phi, b), \cdot \right)\|\le C \|\phi\|_{\LMO^d}\|b\|_{\BMO^d}.
$$
Hence
$$
     \| \Pi(\phi, b)\|_{\BMO^d} \le C \|\phi\|_{\LMO^d}\|b\|_{\BMO^d}.
$$
\end{proof}

In the previous theorem, sharp estimates of $L^2$ norms of restrictions of $\BMO$ functions to rectangles were required. We do not know
such sharp estimates for restrictions of $\BMO$ functions to general open sets:
\begin{question}
By duality,
$$
    |m_\Omega b| \lesssim \| \frac{\chi_\Omega}{|\Omega|}\|_{H^1_d(\T^2)} \|b\|_{\BMO^d},\,\,\, b\in \BMO^d(\T^2)
$$
for all open sets $\Omega$, and this is sharp for each individual
set $\Omega$.

Is it true that ``estimates for $p$-norms are no worse than the estimate for the average $m_{\Omega}b$'', i.~e.
$$
     \|\chi_\Omega b\|_p^p \lesssim \|\frac{\chi_\Omega}{|\Omega|}\|^p_{H^1_d(\T^2)} |\Omega|  \text{ for } 1< p < \infty?
$$
The John-Nirenberg Theorem \cite{ChFef1} gives
$$
     \|\chi_\Omega b\|_p^p \lesssim \log(\frac{4}{|\Omega|})^{2p} |\Omega|,
$$
which is easily seen to be not sharp for certain sets $\Omega$ (for example by considering
``long thin'' rectangles and using Lemma \ref{lemma:avgrowth}).
\end{question}


\section{The other paraproducts}    \label{sec:others}
There are four dyadic paraproducts in two variables, namely the paraproduct $\Pi = \Pi^{(0,0)}$ discussed above, its
adjoint defined by
$$
\Pi^{(1,1)}(\phi,f)=\Delta_\phi f=\Delta(\phi,f) =  \sum_{R \in \RR} \frac{\chi_R}{|R|} \phi_R f_R,
$$ and the mixed paraproducts
given by
$$
\Pi^{(1,0)}(\phi,f)=\sum_{I \times J \in \RR}\frac{\chi_I(s)}{|I|} h_J(t) \phi_{I \times J} m_J f_I,
$$
$$
\Pi^{(0,1)}(\phi,f)= \sum_{I \times J \in \RR} h_I(s)\frac{\chi_J(t)}{|J|} \phi_{I \times J} m_I f_J,
$$
see \cite{bp}. Here and in the following, the variables are sometimes included to explain dependency on the different variables rather than to indicate pointwise equality.

Interestingly, all four paraproducts have a
different boundedness behaviour on $\BMO^d(\T^2)$.
\begin{definition}
Let $\phi \in L^2(\T^2)$. We say that $ \phi \in \LMO_1^d(\T^2)$, if there exists $C >0$ with
$$
    \|Q^{(1)}_{i} \phi\|_{\BMO^d} \le C \frac{1}{i+1 }
$$
for all $i \in \N_0$.

We say that
$ \phi \in \LMO_2^d(\T^2)$, if there exists $C >0$ with
$$
    \|Q^{(2)}_{j} \phi\|_{\BMO^d} \le C \frac{1}{j +1}
$$
for all $j \in \N_0$.  The infimum of such constants is denoted by $\|\phi\|_{\LMO_1^d}$, $\|\phi\|_{\LMO_2^d}$,
respectively.
\end{definition}

\begin{theorem}   \label{thm:others}
Let $\phi \in L^2(\T^2)$. Then
\begin{enumerate}
\item $\Delta_\phi: \BMO^d(\T^2) \rightarrow \BMO^d(\T^2)$ is bounded, if and only if $ \phi \in \BMO^d$. \\
Moreover, $\|\Delta_\phi\|_{\BMO^d \rightarrow \BMO^d} \approx\|\phi\|_{\BMO^d}$.
\item  $\Pi^{(0,1)}(\phi, \cdot): \BMO^d(\T^2) \rightarrow \BMO^d(\T^2)$ is
  bounded, if $ \phi \in \LMO_1^d(\T^2)$.\\
Moreover, $\|\Pi^{(0,1)}(\phi,\cdot)\|_{\BMO^d(\T^2) \rightarrow \BMO^d(\T^2)}
  \lesssim \|\phi\|_{\LMO_1^d(\T^2)}$.
\item  $\Pi^{(1,0)}(\phi,\cdot): \BMO^d(\T^2) \rightarrow \BMO^d(\T^2)$ is
  bounded, if $ \phi \in \LMO_2^d(\T^2)$.\\
Moreover, $\|\Pi^{(1,0)}(\phi,\cdot)\|_{\BMO^d(\T^2) \rightarrow \BMO^d(\T^2)}
  \lesssim \|\phi\|_{\LMO_2^d(\T^2)}$.
\end{enumerate}
\end{theorem}
\begin{proof} (1) was shown in \cite{bp}. To show (2), we will follow
  a simplified version of the ideas of the proof of  Theorem
  \ref{thm:main}.

\begin{lemma} \label{lemma:coreone}  Let $b \in L^2(\T^2)$ and let $k \in \N$. Then
$$
   \|\Pi_b E^{(1)}_{k}\|_{L^2 \to L^2} =   \| \Pi_{\sigma^{(1)}_{k} b}\|_{L^2 \to L^2},
$$
where
$$
   (\sigma^{(1)}_{k}b)_{I,J} =\left\{ \begin{matrix} b_{I,J} & \text{ if } & |I|
   > 2^{-k}\\
 (\sum_{I' \subseteq I} |b_{I',J}|^2)^{1/2}   & \text{ if } &
   |I| = 2^{-k}\\
      0 & \text{ otherwise. }&
                               \end{matrix} \right.
$$
\end{lemma}
\begin{proof} As in Lemma \ref{lemma:core}.
\end{proof}
It remains to prove the following.
\begin{lemma} \label{lemma:core2one}  Let
 $\phi, b \in \BMO^d(\T^2)$, $k \in \N$. Then
$$
   \|\Pi\left(\Pi^{(0,1)}(\phi, b), E^{(1)}_{k}\, \cdot \right)\|_{L^2 \to L^2} \lesssim (k+1) \|\phi\|_{\BMO^d} \|b \|_{\BMO^d}.
$$
\end{lemma}
\begin{proof} We write $E$ for $E^{(1)}$, and $\sigma$ for $\sigma^{(1)}$.
Following the results in Lemma \ref{lemma:coreone}, we estimate
$$
\| \sigma_k( {\Pi^{(0,1)}}(\phi, b))\|_{\BMO^d} \le \| \sigma_k(
  {\Pi^{(0,1)}}({E_k\phi}, b))\|_{\BMO^d}
+ \| \sigma_k(
  {\Pi^{(0,1)}}({Q_k\phi}, b))\|_{\BMO^d}.
$$

We start with the second term.
Since
${\Pi^{(0,1)}}(Q_k\phi, b)$ has no nontrivial Haar terms in the first variable for intervals $I$ with $|I| >2^{-k}$,
$$
    \sigma_k( {\Pi^{(0,1)}}(Q_k\phi, b)) =
       \sum_{J \in \D} \sum_{|I| = 2^{-k}} h_I(s)  (\sum_{I' \subseteq I} |\phi_{I'J}|^2 |m_{I'} b_J|^2)^{1/2}
             \frac{\chi_J}{|J|}(t),
$$
and this has only nontrivial Haar terms in the first variable for intervals $I$ with
$|I|=2^{-k}$. As before, the $\BMO^d$ condition now only has to be considered on rectangles of the form $R=I \times J$, $|I| = 2^{-k}$. Thus
\begin{eqnarray*}
   \| P_R \sigma_k( {\Pi^{(0,1)}}(Q_k\phi, b))\|_2^2
   &=&  \| P_R \sigma_k( {\Pi^{(0,1)}}({P_RQ_k\phi}, b))\|_2^2 \\
& \le &  \| \sigma_k( {\Pi^{(0,1)}}({P_RQ_k\phi}, b))\|_2^2 \\
&=&  \|  {\Pi^{(0,1)}}({P_RQ_k\phi}, b)\|_2^2 \\
&=& \| {\Pi^{(0,1)}}({P_R \phi}, \chi_I(s) P_J b) \|^2_2 \\
&\lesssim& \| P_R \phi\|_{\BMO^d} \| \chi_I P_J b \|_{2}^2\\
        &\lesssim& (k+1)^2 |R|\|\phi\|_{\BMO^d}^2 \|b\|^2_{\BMO^d}. \\
\end{eqnarray*}

Now we have to deal with the first term $ \| \sigma_k(
  {\Pi^{(1,0)}}({E_k\phi}, b))\|_{\BMO^d}$. Let $\Omega
  \subseteq \T^2$ be open and write $\J_I = \cup_{J \in \D, I \times J
  \subseteq \Omega } J $ for $I \in \D$. Then
\begin{eqnarray*}
   \| P_\Omega \left(\sigma_k( {\Pi^{(0,1)}}(\phi, b))\right)\|_2^2
   &=&  \| P_{\O} \sigma_k ({\Pi^{(0,1)}}(P_\Omega E_k\phi, b))\|_2^2 \\
& \le &  \| \sigma_k ({\Pi^{(0,1)}}({P_\Omega E_k\phi}, b))\|_2^2
= \|  {\Pi^{(0,1)}}({P_\Omega E_k\phi}, b)\|_2^2 \\
   &=& \|\sum_{I \in \D, |I| > 2^{-k}} \sum_{J \in \D: I \times J \subseteq \Omega }  h_{I}(s) \frac{\chi_{J}}{|J|}(t)
     \phi_{IJ} m_{I} b_{J} \|_2^2\\
 &=& \sum_{I \in \D, |I| > 2^{-k}} \|\sum_{J \subseteq \J_I } \frac{\chi_{J}}{|J|}(t)
     \phi_{IJ} m_{I}b_{J} \|_2^2\\
 &=& \sum_{I \in \D,  |I| > 2^{-k}} \|\Delta_{m_I b} P_{\J_I} \phi_I \|_2^2\\
&\lesssim& \sum_{I \in \D,  |I| > 2^{-k}} \|m_I b\|^2_{\BMO^d}  \|P_{\J_I} \phi_I \|_2^2\\
&\lesssim& (k+1)^2 \|b\|^2_{\BMO^d} \sum_{I \in \D}  \|P_{\J_I} \phi_I \|_2^2\\
&\lesssim& (k+1)^2 \|b\|^2_{\BMO^d}\|P_\Omega \phi \|_2^2\\
&\lesssim&  (k+1)^2 \|b\|^2_{\BMO^d} \|\phi\|^2_{\BMO^d}|\Omega|
\end{eqnarray*}
by Lemma \ref{lemma:avgrowth}.
\end{proof}
As in the last section, we immediately deduce
\begin{equation}
  \|\Pi\left(\Pi^{(0,1)}(Q^{(1)}_j \phi, b), E^{(1)}_{k}\, \cdot \right)\|_{L^2 \to L^2} \lesssim \frac{k+1}{j+1} \|\phi\|_{\LMO_1^d} \|b \|_{\BMO^d}.
\end{equation}

The remainder of the proof of (2) is now exactly analogous to the proof of
Therem \ref{thm:main}, defining $T_N = \Pi({\Pi^{(0,1)}}(\phi, \cdot), P_N\cdot)$, where
  $P_{N} = \sum_{i=2^N-1}^{2^{N+1}-2}
   \Da^{(1)}_{i}$,
and using Cotlar's Lemma in one parameter. Finally, (3) follows by simply
switching variables.
\end{proof}


\section{Generalization to more than two variables}   \label{sec:multi}
The results of Section \ref{sec:main} and \ref{sec:others} generalize easily to more that two variables. We will just state
the results here, the proofs
are very similar to those in the previous sections.

Let $N \in \N$, let $\RR=\{R=R_1\times \cdots \times R_N\in \T^N: R_j\in \D\}$ the $N$-fold Cartesian product of
the set of dyadic intervals $\D$, and let $(h_R)_{R \in \R}$ denote the corresponding product Haar basis of $L_0^2(\T^N)$. Recall that a function  $b \in L^2(\T^N)$
is in the dyadic product BMO space $\BMO^d(\T^N)$, if
$$
   \|b\|^2_{\BMO^d}= \sup_{\Omega \in \T^N \text{open}} \frac{1}{|\Omega|}\sum_{R \subset \Omega} |b_R|^2 < \infty.
$$

\begin{definition} \label{def:multLMO}
Let $\phi \in L^2(\T^N)$, $\vec {\delta}=(\delta_1,\cdots,\delta_N)$, $\delta_j\in \{0,1\}$. Then $ \phi \in \LMO_{\vec {\delta}}^d(\T^N)$, if and only if there exists $C >0$ such that for each
dyadic rectangle $R= R_1 \times R_2 \times \cdots \times R_N \in \D^N$ and each open set $\Omega \subseteq R $,
$$
   \frac{\log(\frac{4}{|R_{\delta_1}|})^2 \cdots  \log(\frac{4}{|R_{\delta_N}|})^2}{|\Omega|} \sum_{Q \in \RR, Q \subseteq \Omega} |\phi_Q|^2 \le C,
$$
where
$$
R_{\delta_j}=\left\{ \begin{matrix} R_j &\text{if}& \delta_j=0\\
                    \T &\text{otherwise.}&
                    \end{matrix}\right.
$$
\end{definition}

When $\vec {\delta}=\vec{0}=(0,\cdots,0)$, $\LMO_{\vec {\delta}}^d(\T^N)$ corresponds to the generalization of $\LMO^d(\T^2)$ and is denoted by
$\LMO^d(\T^N)$. One easily sees that for
$\vec {\delta}=(1,\cdots,1)= \vec{1}$, the corresponding space is just the space $\BMO^d(\T^N)$.

As before, we consider paraproducts as defined in (\ref{paraprodgene2}) for triples $(\veps, \vda, \vba)$ with $\vec {\varepsilon}=(0,\cdots,0)$, and
$$
\delta_j=\left\{ \begin{matrix} 1 &\text{if}& \beta_j=0\\
                    0 &\text{otherwise.}&
                    \end{matrix}\right.
$$

For simplicity, we write
\begin{equation}\label{paraprodgeneN}\Pi^{\vec {\beta}}_\phi f=B_{\vec {\varepsilon},\vec {\delta},\vec {\beta}}(\phi,f):=\sum_{R\in \RR}\phi_R \langle f,h^{\vec {\delta}}_R\rangle h^{\vec {\beta}}_R.\end{equation}

As in the case $N=2$, $\Pi$ and $\Delta$ are given by $\Pi^{\vba}$ for $\vec {\beta}=(0,\cdots,0)$ and $\vec {\beta}=(1,\cdots,1)$, respectively.

Here is the result on the boundedness of $\Pi^{\vba}_\phi$ on $\BMO^d(\T^N)$.
 \begin{theorem}   \label{thm:multothers}
Let $\phi \in L_0^2(\T^N)$, $\vec {\beta}=(\beta_1,\cdots,\beta_n)$, $\beta_j\in \{0,1\}$. Then

\begin{enumerate}
\item $\Pi^{(0, \dots,0)}_\phi = \Pi_\phi:\BMO^d(\T^N) \rightarrow \BMO^d(\T^N)$ is bounded, if and only if $ \phi \in \LMO^d(\T^N)$.
Moreover, $$\|\Pi_\phi\|_{\BMO^d \rightarrow \BMO^d} \approx\|\phi\|_{\LMO^d}.$$
\item $\Pi^{(1,\dots,1)}_\phi = \Delta_\phi: \BMO^d(\T^N) \rightarrow \BMO^d(\T^N)$ is bounded, if and only if $ \phi \in \BMO^d(\T^N)$.
Moreover, $$\|\Delta_\phi\|_{\BMO^d \rightarrow \BMO^d} \approx\|\phi\|_{\BMO^d}.$$
\item For $\vec {\beta}\neq (0,\cdots,0),  (1,\cdots,1)$,\\
$\Pi^{\vec {\beta}}_\phi:\BMO^d(\T^N) \rightarrow \BMO^d(\T^N)$ is bounded if $ \phi \in \LMO_{\vec {\beta}}^d(\T^N)$.
Moreover, $$\|\Pi^{\vec {\beta}}\phi\|_{\BMO^d(\T^N) \rightarrow \BMO^d(\T^N)}
  \lesssim \|\phi\|_{\LMO_{\vec {\beta}}^d}.$$
\end{enumerate}

\end{theorem}


\section{Hankel operators, commutators and dyadic shifts}      \label{sec:hankel}


In this section, we are interested in application of the previous results to boundedness of iterated commutators with Hilbert transform on $H^1(\R^N)$, i.e the predual of $\BMO(\R^N)$ as defined by A. Chang and R. Fefferman \cite{ChFef1}. Again, for simplicity, we restrict our presentation to the two dimensional case, as the general case follows the same way.

We will be writing $H_1$ and $H_2$ for the Hilbert transform in the first and second variable respectively. Let us first recall the following result.
\begin{theorem}[\cite{ChFef1, fergsad, ferglac}] Let $b
  \in \BMO(\R^2)$. Then
$$
   [H_1, [H_2, \phi]]: L^2(\R^2) \rightarrow L^2(\R^2)
$$
is bounded, and $\|[H_1, [H_2, \phi]]\|_{L^2 \rightarrow L^2} \approx
\|\phi\|_{\BMO}$.
\end{theorem}

One would like to characterize boundedness of commutators on the endpoint spaces $H^1(\R^2)$ and $\BMO(\R^2)$ in an analogous fashion  in terms of a suitable notion of
$\LMO(\R^2)$. However, this is not possible even in one parameter (see e.~g.~\cite{JPS}, Remark 4.1).
The reason is the slow decay of the kernel $1/x$  of the Hilbert transform. This means that for  $\phi$ with compact support and $f\in H^1(\R)$ an atom,
  $\phi Hf$ will be integrable, but $\phi f$ will in general not have average zero, so $H(\phi f)$ behaves like $1/x$ at $\infty$ and is therefore not integrable at $\infty$.
  On the dual side, this amounts to saying that one should only consider  functions with compact support in $\BMO(\R^2)$.

   In terms of our main result on paraproducts Theorem \ref{thm:main}, one sees easily that there is no good estimate for averages of $\BMO^d(\R)$ functions, and the theorem does not hold for $\BMO^d(\R^2)$.
   However, we will prove a local estimates for paraproducts and commutators. Our tools are adapated for the case of $\R^N$. It would be interesting to see the result in
   the case of the polydisc.

 Let
  \begin{equation}\label{BMOrestrict}
 \BMO([0,1]^2):=\{f\in \BMO(\R^2): \supp f\subseteq [0,1]^2\}.
 \end{equation}
and
$$
\BMO^d([0,1]^2) = \{ f \in \BMO^d(\R^2): \supp f\subseteq [0,1]^2\}.
$$
 We say that $ f \in \LMO^d([0,1]^2) $, if $f \in \BMO^d([0,1]^2)$ and there exists $C>0$ with
$$
   \|Q_{\vk} f\|_{\BMO^d(\R^2)} \le C\frac{1}{(k_1+1)(k_2+1)} \text{ for } \vk=(k_1, k_2) \in \N_0 \times \N_0.
$$
The spaces  $\LMO_{\vba}^d([0,1]^2) $ are defined correspondingly.
Here are our local estimates  on paraproducts.

 \begin{theorem}   \label{thm:locmultothers}
Let $\phi \in L_0^2([0,1]^2)$, $\vec {\beta}=(\beta_1,\beta_2)$, $\beta_j\in \{0,1\}$. Then

\begin{enumerate}
\item $\Pi^{(0, 0)}_\phi = \Pi_\phi:\BMO^d([0,1]^2) \rightarrow \BMO^d(\R^2)$ is bounded, if and only if $ \phi \in \LMO^d([0,1]^2)$.
Moreover, $$\|\Pi_\phi\|_{\BMO^d([0,1]^2) \rightarrow \BMO^d(\R^2)} \approx\|\phi\|_{\LMO^d([0,1]^2)}.$$
\item $\Pi^{(1,1)}_\phi = \Delta_\phi: \BMO^d([0,1]^2) \rightarrow \BMO^d(\R^2)$ is bounded, if and only if $ \phi \in \BMO^d([0,1]^2)$.
Moreover, $$\|\Delta_\phi\|_{\BMO^d([0,1]^2) \rightarrow \BMO^d(\R^2)} \approx\|\phi\|_{\BMO^d([0,1]^2)}.$$
\item For $\vec {\beta}\neq (0,0),  (1,1)$,\\
$\Pi^{\vec {\beta}}_\phi:\BMO^d([0,1]^2) \rightarrow \BMO^d(\R^2)$ is bounded if $ \phi \in \LMO_{\vec {\beta}}^d([0,1]^2)$.
Moreover, $$\|\Pi^{\vec {\beta}}\phi\|_{\BMO^d([0,1]^2) \rightarrow \BMO^d(\R^2)}
  \lesssim \|\phi\|_{\LMO_{\vec {\beta}}^d([0,1])}^2.$$
\end{enumerate}

\end{theorem}
The theorem relies only on the appropriate version of Lemma \ref{lemma:avgrowth}, and a slight change of the decomposition of the identity in the proofs of Theorem
\ref{thm:main} and \ref{thm:others}.

\begin{lemma}  \label{lemma:locgrowth} For a bounded (not necessarily dyadic) interval $I \subset \R$, let
$$
         s(I) =  \left\{ \begin{matrix} {\log|I|^{-1} +1}  & \text{ for } & |I| \le 1 \\
                                      1   & \text{ for } & |I| >1. \end{matrix}   \right.
$$
For a bounded (not necessarily dyadic) axis-parallel rectangle $R = I \times J \subset \R^2$, let $s(R)=s(I) s(J)$.

 Then for each $b \in \BMO([0,1]^2)$ and each rectangle $R = I \times J \subset \R^2$,
 \begin{enumerate}
\item
 $$
  |m_R b| \lesssim s(R) \|b\|_{\BMO(\R^2)}; 
$$
\item
$$
  \|m_I b\|_{\BMO(\R)} \lesssim s(I) \|b\|_{\BMO(\R^2)}; 
$$
\item
$$
    \|\chi_R b \|_2^2 \lesssim s(R)^2 |R| \|b\|^2_{\BMO(\R^2)}; 
$$
\item
$$
    \|\chi_I P_J b \|_2^2 \lesssim s(I)^2  |I| |J|  \|b\|^2_{\BMO(\R^2)} 
$$
\item
$$
   \|\chi_I m_J b \|_2^2 \lesssim s(J)^2 s(I)^2  |I|  \|b\|^2_{\BMO(\R^2)}
$$
\end{enumerate}
and this is sharp. Here, $P_Jb(s,t) = \chi_J(t) (b(s,t) - m_J b(s))$.
\end{lemma}

Before proving this lemma, let us first turn to the relation of $\BMO([0,1]^2)$, $\BMO(\R^2)$, $\BMO(\T^2)$,  and $\BMO^d(\R^2)$.

First let us consider the relation between $\BMO(\R^2)$  and $\BMO^d(\R^2)$.
This was clarified only quite recently in  \cite{pipher}, \cite{treil}. Given $\alpha=(\alpha_j)_{j\in \mathbb {Z}}\in \{0,1\}^{\mathbb {Z}}$ and
$r\in [1,2)$, we denote by $\mathcal {D}^{\alpha,r}=r\mathcal {D}^{\alpha}$ the dilated and translated standard dyadic grid $\mathcal {D}$ of $\R$ in the sense of \cite{hyt}. For
$\vec {\alpha}=(\alpha^1,\alpha^2)\in \{0,1\}^{\mathbb {Z}}\times \{0,1\}^{\mathbb {Z}}$ and $\vec {r}=(r_1,r_2)\in [1,2)^2$,
we define $\mathcal {D}^{\vec {\alpha},\vec {r}}$ to be the dilated and translated product dyadic grid in $\R^2$. That is
$Q=Q_1\times Q_2\in \mathcal {D}^{\vec {\alpha},\vec {r}}$ if $Q_1\in r_1\mathcal {D}^{\alpha^1}$ and $Q_2\in r_2\mathcal {D}^{\alpha^2}$.
The work in \cite{pipher}, \cite{treil}  implies in particular that
\begin{multline*}
\BMO(\R^2)=\bigcap_{\vec {\alpha}\in \{0,1\}^{\mathbb {Z}}\times \{0,1\}^{\mathbb {Z}}, \vec {r}\in [1,2)^2 }\BMO^{d,\vec {\alpha}, \vec {r}}(\R^2)   \\
   =  \bigcap_{\vec {\alpha}\in \{0,1\}^{\mathbb {Z}}\times \{0,1\}^{\mathbb {Z}} }\BMO^{d,\vec {\alpha}, \vec {r_0}}(\R^2)  \text{ for any } \vec{r_0} \in [0,1)^2 ,
\end{multline*}
where $\BMO^{d,\vec {\alpha}, \vec {r}}(\R^2)$ is the dyadic $\BMO(\R^2)$ defined with respect to the product dyadic grid $\mathcal {D}^{\vec {\alpha},\vec {r}}$.
One also obtains that
$$
\BMO([0,1]^2)=\bigcap_{\vec {\alpha}\in \{0,1\}^{\mathbb {Z}}\times \{0,1\}^{\mathbb {Z}}, \vec {r}\in [1,2)^2 }\BMO^{d,\vec {\alpha}, \vec {r}}([0,1]^2).
$$

Now  let us consider the relationship between $\BMO([0,1]^2)$ and $\BMO(\T^2)$. It is easy to see that under the usual identification of $[0,1)$ and $\T$,
$\BMO([0,1]^2) \neq \BMO(\T^2)$. A simple example is the function
$$
    b(s,t) =  \log( \min(s, 1-s)) \cdot
\log( \min(t, 1-t)) \quad (s,t) \in [0,1]^2,
$$
 which is in $\BMO(\T^2)$, but
not in $\BMO([0,1]^2)$.

On the other hand, for each $a<0$, $b>1$, we can extend $b \in \BMO([0,1]^2)$ to a doubly
$(b-a)$-periodic function in $\BMO(\R^2)$, by first considering
it as a function on $[a,b)^2$ and then extending this function doubly periodically. Of course the subspace of doubly 1-periodic functions in $\BMO(\R^2)$ can be identified
with $\BMO(\T^2)$, and for any $b>a$, the subspace of doubly $b-a$ periodic functions in $\BMO(\R^2)$ can be identified with $\BMO(\T^2)$ by means of an appropriate dilation.

\proof of Lemma \ref{lemma:locgrowth}. The proof follows mostly from Lemma \ref{lemma:avgrowth}, but we have to attend to a few technicalities.
Let $R= I \times J$ be a rectangle. We only need to consider the case that $I \times J \cap [0,1]^2 \neq \emptyset$, that is $I \cap [0,1] \neq \emptyset$ and $J \cap [0,1] \neq \emptyset$. If $|I|, |J| \le 1$, then $R \cup [0,1]^2 \subset [-1,2]^2$. By first considering the functions in $\BMO([0,1]^2)$ as functions on $[-1,2)$, then extending
doubly periodically with period 3 and identifying  the space of doubly periodic functions with period 3
 in $\BMO(\R^2)$ with $\BMO(\T^2)$, we can apply Lemma \ref{lemma:avgrowth} to
obtain the desired estimates (1) - (5). Note that (5), which didn't appear explicitly in Lemma \ref{lemma:avgrowth}, is a simple consequence of (2) and (3), applied for the one-dimensional case.

 Then (1) is obtained in general by writing $I'=I \cap [0,1]$, $J'=J \cap [0,1]$ and observing that $|I'|s(I'), |J'|s(J')<2$, which yields
\begin{multline*}
|m_{I\times J}b|=\frac{|I'|}{|I|}\frac{|J'|}{|J|}|m_{I'\times J'}b|  \le  \frac{|I'|}{|I|}\frac{|J'|}{|J|} s(I') s(J') \| b \|_{\BMO}  \\
                  \lesssim s(I) s(J) \| b \|_{\BMO}       \,\,\,\textrm{for}\,\,\,|I|,|J|>1
\end{multline*}
and
 \begin{multline*}
        |m_{I\times J}b|=\frac{|I'|}{|I|}|m_{I'\times J}b|   \le  \frac{|I'|}{|I|}  s(I') s(J) \| b \|_{\BMO}  \\
          \lesssim  s(I) s(J) \| b \|_{\BMO}         \,\,\,\textrm{for}\,\,\,|I|>1\,\,\,\textrm{and}\,\,\,|J|\le 1.
 \end{multline*}

To get estimate (2) in case $|I | \le 1 $,
 we need to check boundedness of $\frac{1}{|J|^{1/2}} \| P_J m_I b\|_2$ for arbitrary intervals $J$ with $J \cap [0,1] \neq \emptyset$. For $|J|\le 1$,
we get the desired estimate as above. For $|J|>1$, write $J' = J \cap [0,1]$, $J'' = J \cap [0,1]^c$ and obtain
\begin{equation}   \label{eq:Jproj}
   \begin{split}
     &   P_J m_I b(t)        \\
      = & \, \chi_J(t) (m_I b(t) - m_{I \times J} b)    \\
       = & \, \chi_{J'}(t)\left(m_I b(t)  -  \frac{|J'|}{|J|} m_{I \times J'} b\right) - \chi_{J''}(t)     \frac{|J'|}{|J|} m_{I \times J'} b   \\
        =&\,  \chi_{J'}(t)(m_I b(t)  -   m_{I \times J'} b) +  \chi_{J'}(t) \frac{|J''|}{|J|} m_{I \times J'} b    - \chi_{J''}(t)     \frac{|J'|}{|J|} m_{I \times J'} b. \\
       \end{split}
       \end{equation}
Thus
\begin{multline*}
 \| P_J m_I b(t)  \|_2   \le   \| \chi_{J'} (m_I b(t)  - m_{I \times J'} b) \|_2   \\ +   \| \chi_{J'}(t)   \frac{|J''|}{|J|}   m_{I \times J'} b \|_2 +  \| \chi_{J''}(t)  \frac{|J'|}{|J|} m_{I \times J'} b \|_2   .
\end{multline*}
The first summand is estimated by the previous argument for the case $|J| \le 1$. For the second and third summand, we observe that by (1),
$$|m_{I \times J'} b|\lesssim s(I)s(J')\|b\|_{\BMO}$$
and consequently, as $|J'|\le 1$,
$$\| \chi_{J'}(t)   \frac{|J''|}{|J|}   m_{I \times J'} b \|_2 \lesssim \frac{|J'|^{1/2} |J''|}{|J|}s(I) s(J') \|b\|_{\BMO} \lesssim    s(I)  \|b\|_{\BMO}     $$
and

$$\| \chi_{J''}(t)  \frac{|J'|}{|J|} m_{I \times J'} b \|_2\lesssim \frac{|J''|^{1/2} |J'|}{|J|}s(I) s(J') \|b\|_{\BMO}   \lesssim    s(I)  \|b\|_{\BMO}.$$

Now consider the case $|I|>1$, $I \cap [0,1]  \neq \emptyset$.  Similarly to the above, let $I' = I \cap [0,1]$.
 Writing $m_I b =\frac{|I'|}{|I|}  m_{I'} b $, we obtain the same result.

It will be useful for the further estimates to prove (5) at this point. It is clear for $|I|, |J| \le 1$, otherwise
\begin{multline*}
    \|\chi_I m_J b \|_2^2  =   \frac{|J'|^2}{|J|^2}      \|\chi_{I'} m_{J'} b \|_2^2  \lesssim  \frac{|J'|^2}{|J|^2}  s(J')^2     | I' |  s(I')^2  \|b\|^2_{\BMO} \\
          \lesssim  s(J)^2 s(I)^2 |I|   \|b\|^2_{\BMO} .
\end{multline*}

For (3), write for $R = I \times J$
$$
      \chi_R b = P_R b + \chi_R m_I  b + \chi_R m_J b - \chi_R m_R b
$$
and use (1) and (5).

For (4), write
$$
       \chi_I(s) P_J b(s,t)    = \chi_I(s) \chi_J(t) b(s,t)  - \chi_I(s) m_J b(s)
$$
and use (3) and (5).
\qed

\proof of Theorem \ref{thm:locmultothers}. We only prove the assertion $(1)$. The proof for the other paraproducts
 uses the same ideas combined with those in the proof of Theorem \ref{thm:others}.

 We want to prove that given $\phi\in \LMO^d([0,1]^2)$, $b\in \BMO^d([0,1]^2)$ and  $f\in L^2(\mathbb R^2)$, the function
 $\Pi\left(\Pi(\phi,b),f\right)$ belongs to $L^2(\mathbb R^2)$, with the appropriate norm estimate.

 We now work with the standard system $\D(\R)$ of dyadic intervals in $\R$, the Haar basis
  $(h_I \otimes h_J)_{I, J \in \D(\R)} =   (h_R)_{R \in \D(\R) \times \D(\R)} $ of $L^2(\R^2)$, and the
  decomposition
 $$
      f = \sum_{j_1 = - \infty}^\infty   \sum_{j_2 = - \infty}^\infty  \Da_{\vj} f                                        ,
 $$
 where
 \begin{multline*}
        \Da_{\vj} f = \sum_{|I |= 2^{-j_1}, |J|= 2^{- j_2} }   h_I(s) h_J(t) \langle f, h_I \otimes h_J \rangle    \\
           = \sum_{ R \in \D_{j_1}(\R) \times \D_{j_2}(\R)}   h_R  \langle f, h_R \rangle      \text{ for }  j_1, j_2 \in \ZZ.
 \end{multline*}

Then
$$T:=\Pi\left(\Pi(\phi,b), \cdot \right)=P_{(0,1)^2}T     +P_{(0,1)\times (0,1)^c}T   +  P_{(0,1)^c\times (0,1)} T   +P_{(0,1)^c\times (0,1)^c}T   ,$$
where
\begin{eqnarray*}
 P_{(0,1)\times (0,1)}&=&   \sum_{j_1 = 0}^{\infty}\sum_{j_2=0}^{\infty}      \Da_{\vj},  \\
  P_{(0,1)\times (0,1)^c}&=&   \sum_{j_1 =0}^{\infty}\sum_{j_2=-\infty}^{-1}      \Da_{\vj}, \\
   P_{(0,1)^c\times (0,1)}&=&   \sum_{j_1 = - \infty}^{-1}\sum_{j_2=0}^{\infty}      \Da_{\vj}, \\
 P_{(0,1)^c\times (0,1)^c}&=&   \sum_{j_1 = - \infty}^{-1}\sum_{j_2=-\infty}^{-1}      \Da_{\vj}.
\end{eqnarray*}

Hence we only need to check the $L^2$-boundedness of each of the four terms in the right hand side of the above identity.

The estimate for
$$
      P_{(0,1)^2}\Pi\left(\Pi(\phi,b),f\right)   =      \Pi\left(\Pi(P_{(0,1)^2}       \phi,b),f\right)
$$
  for  $f\in L^2(\mathbb R^2)$ is
 obtained exactly as in the proof of Theorem \ref{thm:main}, with the help of the growth estimate in Lemma \ref{lemma:locgrowth}.

 For the fourth term, we observe that with
$$
P_{(0,1)^c\times (0,1)^c}\Pi\left(\Pi(\phi,b), \cdot \right)=\Pi\left(\Pi(P_{(0,1)^c\times (0,1)^c}\phi,b), \cdot \right),
$$
we only have to check that for given $\phi\in \LMO^d([0,1]^2)$ and $b\in \BMO^d([0,1]^2)$, $P_{(0,1)^c\times (0,1)^c}\Pi(\phi,b)$ belongs to $\BMO^d(\mathbb R^2)$. Using the fact that for $R=I\times J\in \mathcal R$ with $|I|,|J|\ge 1$, $|m_Rb|\lesssim \|b\|_{\BMO}$, one obtains directly that for any open set $\Omega\subset \mathbb R^2$,
$$\|P_\Omega\left(P_{(0,1)^c\times (0,1)^c}\Pi(\phi,b)\right)\|_2^2\le \|P_\Omega\phi\|_2^2\|b\|_{\BMO}^2,$$
which proves that this term is bounded on $L^2(\mathbb R^2)$.

As the second and third terms are symmetric, we only prove the boundedness of the second one. For this, we need to go back to the proof of Theorem \ref{thm:others}. Again, we use  that
$$P_{(0,1)\times (0,1)^c}\Pi\left(\Pi(\phi,b), \cdot \right)=\Pi\left(\Pi( P_{(0,1)\times (0,1)^c}      \phi,b), \cdot \right).$$
\begin{lemma}\label{lemma:technicalend}
Let $\phi,b\in \BMO^d(\T^2)$, $k\in \N_0$. Then
$$\|\Pi\left(\Pi(P_{(0,1)\times (0,1)^c}\phi,b),E_k^{(1)}\right)\|_{L^2\rightarrow L^2}\lesssim (k+1)\|\phi\|_{\BMO^d}\|b\|_{\BMO^d}.$$
\end{lemma}
\proof
We follow the proof of Lemma \ref{lemma:core2one}. Again we write $E$ for $E^{(1)}$, and $\sigma$ for $\sigma^{(1)}$. As in
 Lemma \ref{lemma:coreone}, we need to estimate
\begin{eqnarray*}
&&\|\sigma_k\left(\Pi(P_{(0,1)\times (0,1)^c}\phi,b)\right)\|_{\BMO^d}\\
     &\le & \|\sigma_k\left(\Pi(P_{(0,1)\times (0,1)^c}E_k\phi,b)\right)\|_{\BMO^d} + \|\sigma_k\left(\Pi(P_{(0,1)\times (0,1)^c}Q_k\phi,b)\right)\|_{\BMO^d} \\
    & = &    \| \left(\Pi(P_{(0,1)\times (0,1)^c}E_k\phi,b)\right)\|_{\BMO^d} + \|\sigma_k\left(\Pi(P_{(0,1)\times (0,1)^c}Q_k\phi,b)\right)\|_{\BMO^d} .\\
\end{eqnarray*}
Starting we the first term, we obtain for any open set $\Omega\subset \mathbb R^2$,
\Beas
\frac{1}{|\Omega|}\|P_\Omega\left(\Pi(P_{(0,1)\times (0,1)^c}E_k\phi,b)\right)\|_2^2 &=& \frac{1}{|\Omega|}\sum_{R=I\times J,|I|>2^{-k},|J|>1, R\subset \Omega}|\phi_R|^2|m_Rb|^2\\ &\lesssim& \frac{(k+1)^2}{|\Omega|}\sum_{R=I\times J,|I|>2^{-k},|J|>1, R\subset \Omega}|\phi_R|^2\\ &\lesssim& (k+1)^2\|\phi\|_{\BMO^d}^2\|b\|_{\BMO^d}^2.
\Eeas
where we use Lemma \ref{lemma:locgrowth}.

For the second term, we observe that $\sigma_k\left(\Pi(P_{(0,1)\times (0,1)^c}Q_k\phi,b)\right)$ has only nontrivial coefficients for those rectangles $R=I\times J$ with $|I|=2^{-k}$ and $|J|>1$. Hence, it is enough to check the $\BMO$-norm on rectangles $R=I\times J$ with $|I|=2^{-k}$ and $|J|>1$. We obtain
\Beas
\frac{1}{|R|}\|P_R\sigma_k \left(\Pi(P_{(0,1)\times (0,1)^c}Q_k\phi,b)\right)\|_2^2 &=& \frac{1}{|R|}\sum_{J'\subseteq J}|\left(\sigma_k\left(\Pi(P_{(0,1)\times (0,1)^c}Q_k\phi,b)\right)\right)_{I,J'}|^2\\
&=& \frac{1}{|R|}\sum_{I'\subseteq I}\sum_{J'\subseteq J}     |\phi_{I'J'}|^2    |m_{I'J'}b|^2 \\
&=& \frac{1}{|R|}\|\Pi(P_R\phi,\chi_Rb)\|_2^2\\
&\lesssim& \frac{1}{|R|}\|P_R\phi\|_{\BMO^d}\|\chi_Rb\|_2^2\\ &\lesssim& (k+1)^2\|\phi\|_{\BMO^d}^2\|b\|_{\BMO^d}^2.
\Eeas
\qed

The remainder of the proof of boundedness of $P_{(0,1)\times (0,1)^c}\Pi\left(\Pi(\phi,b), \cdot \right)$ follows now with Cotlar's Lemma exactly as in the proof of  Theorem \ref{thm:others}.

\qed

Before giving our main result of this section, we introduce the space $\LMO([0,1]^2)$.
\begin{definition}\label{LMOrestrict}
Let $f\in L^2(\R^2)$. We say that $f\in \LMO([0,1]^2)$ if $\supp f\subseteq [0,1]^2$,
 and there exists a constant $C>0$ such that for any $\vec {\alpha}\in \R\times \R$, $\vec {r}\in [1,2)^2$, and
$\vj = (j_1, j_2) \in \N_0\times \N_0$,
$$
\|Q^{\vec {\alpha}, \vec {r}}_{\vj}f\|_{\BMO^{d,\vec {\alpha}, \vec {r}}([0,1]^2)}\le C\frac{1}{(j_1+1)(j_2+1)}.
$$
\end{definition}
Here, $Q^{\vec {\alpha}, \vec {r}}_{\vj}$ denotes the projection as in (\ref{eq:qdef}), but relative to the dyadic grid $\mathcal{D}^{\vec {\alpha},\vec {r}}$. More precisely, $$Q^{\vec {\alpha}, \vec {r}}_{\vj}f(s,t)=\sum_{r_1|I|\le 2^{-j_1}, r_2|J|\le 2^{-j_2}}\langle f,h_I^{\alpha_1,r_1}h_J^{\alpha_2,r2}\rangle h_I^{\alpha_1,r_1}(s)h_J^{\alpha_2,r2}(t),$$ where $h_I^{\alpha_l,r_l}$ is the Haar wavelet adapted to $I\in r_l\mathcal {D}^{\alpha_l}$, $l=1,2$.

Clearly, $\LMO([0,1]^2)$ continuously embeds into $\BMO([0,1]^2)$.
Moreover, if we denote by $\LMO^{d,\vec {\alpha}, \vec {r}}([0,1]^2)$ the subset of
$\BMO^{d,\vec {\alpha}, \vec {r}}([0,1]^2)$ of functions $f$ such that there exists $C>0$ with
$$
    \|Q_{\vj}^{\vec {\alpha}, \vec {r}}f\|_{\BMO^{d,\vec {\alpha}, \vec {r}}([0,1]^2)}\le C\frac{1}{(j_1+1)(j_2+1)}   \text{ for any }   \vj\in \N_0\times \N_0,
$$
then of course
 $$
 \LMO([0,1]^2)=\bigcap_{\vec {\alpha}\in \{0,1\}^{\mathbb {Z}}\times \{0,1\}^{\mathbb {Z}}, \vec {r}\in [1,2)^2 }\LMO^{d,\vec {\alpha}, \vec {r}}([0,1]^2).
 $$
$\LMO_1([0,1]^2)$ and $\LMO_2([0,1]^2)$ along with their dyadic counterparts are defined analogously.

Here is the main result of this section.

\begin{theorem} \label{thm:hankel} Let $\phi
  \in \LMO([0,1]^2)$. Then
 $$
   [H_1, [H_2, \phi]]: \BMO([0,1]^2) \rightarrow \BMO(\R^2),
$$
is bounded, and $\|[H_1, [H_2, \phi]]\|_{\BMO([0,1]^2) \to \BMO(\R^2)} \lesssim
\|\phi\|_{\LMO([0,1]^2)}$.
\end{theorem}


\begin{proof} of Theorem \ref{thm:hankel}. We use the representation of the Hilbert transform as
  averages
of dyadic shifts from \cite{pet}, \cite{hyt}. Let $S: L^2( \R) \rightarrow
L^2(\R)$ be the bounded linear operator defined by $S h_I = h_{I^+} - h_{I^-}$, $I \in \D$.
Define $S^{(1)} = S \otimes \eins$, $S^{(2)} = \eins \otimes S$,
as operators on $L^2(\R^2) = L^2(\R) \otimes L^2(\R)$. For the
averaging technique, we need to investigate the iterated
commutator
$$
   [S^{(1)}, [S^{(2)},\phi]].
$$
We first prove the following dyadic analogue of the commutator theorem.
\begin{theorem} \label{thm:dyshift} Let $\phi \in \LMO^d([0,1]^2)$. Then
$$
   [S^{(1)}, [S^{(2)},\phi]]: \BMO^d([0,1]^2) \rightarrow \BMO^d(\R^2)
$$
is bounded, and $\|[S^{(1)}, [S^{(2)},\phi]]\|_{\BMO^d([0,1]^2) \rightarrow
  \BMO^d(\R^2)} \lesssim \|\phi\|_{\LMO^d([0,1]^2)}$.
\end{theorem}
\begin{proof}
We formally decompose the multiplication operator with $\phi$ into
$9$ parts: ${\Pi}_\phi$, ${\Delta}_\phi$,
${\Pi^{(0,1)}}_\phi$, ${\Pi^{(1,0)}}_\phi$, ${R_\Delta}_\phi$,
${R_\Pi}_\phi$, ${\Delta_R}_\phi$, ${\Pi_R}_\phi$, ${R_R}_\phi$,
corresponding to the matrix elements $\langle M_\phi h_I(s)
h_J(t), h_{I'}(s) h_{J'}(t) \rangle$ for $I' \subset I$, $I' = I$,
$I' \subset I$, $I' \supset I$, $J' \subset J$, $J' = J$, $J'
\supset J$. Notice that the operator $R$ denotes the Haar-diagonal
part of the multiplication operator.

It is easy to see that $S^{(1)}$ and $S^{(2)}$ are bounded on
$\BMO^d(\R^2)$. Thus, after considering Theorems \ref{thm:main},
\ref{thm:others} and symmetry of variables, we are left to
consider $ [S^{(1)},
  [S^{(2)},{R_R}_\phi]]$,  $ [S^{(1)},
  [S^{(2)},{\Pi_R}_\phi]]$,  $ [S^{(1)},
  [S^{(2)},{\Delta_R}_\phi]]$.

We recall that
$$R_{R_\phi}b(s,t)=\sum_{I,J}b_{IJ}m_{IJ}(\phi)h_I(s)h_J(t),$$
$$\Pi_{R_\phi}b(s,t)=\sum_{I,J}m_J(\phi_I)m_I(b_J)h_I(s)h_J(t)$$ and
       $$\Delta_{R_\phi}b(s,t)=\sum_{I,J}m_J(\phi_I) b_{I,J} h_I(s)h_J^2(t).$$

We start with $ [S^{(1)},
  [S^{(2)},{R_R}_\phi]]$. One verifies that
\begin{multline*}
[S^{(1)},
  [S^{(2)},{R_R}_\phi]] h_{IJ} = (m_{I,J} \phi - m_{I^+,J} \phi -
m_{I,J^+} \phi + m_{I^+, J^+} \phi ) h_{I^+, J^+}\\
  + (m_{I,J} \phi - m_{I^+,J} \phi -
m_{I,J^-} \phi + m_{I^+, J^-} \phi ) h_{I^+, J^-} \\ + (m_{I,J}
\phi - m_{I^-,J} \phi -
m_{I,J^-} \phi + m_{I^-, J^+} \phi ) h_{I^-, J^+} \\
+ (m_{I,J} \phi - m_{I^-,J} \phi - m_{I,J^-} \phi + m_{I^-, J^-}
\phi ) h_{I^-, J^-}.
\end{multline*}
Thus $[S^{(1)},
  [S^{(2)},{R_R}_\phi]]$ preserves the orthogonality of the Haar
  system $(h_{I,J})_{I,J, \in \D}$.  Letting $\tilde \phi = E_{k+1,
  l+1} \phi$ for $|I|= 2^{-k}$, $|J|= 2^{-l}$, we find that
\begin{multline*}
\|[S^{(1)},[S^{(2)},{R_R}_\phi]] h_{IJ}\|_2^2  = \frac{1}{|I||J|}
\|\tilde \phi(s,t) - m_I \tilde
  \phi(t) - m_J \tilde \phi (s) + m_{I \times J} \tilde \phi \|_2^2 \\ =
  \frac{1}{|I||J|} \|P_{I\times J} \tilde \phi \|^2_2  \le \frac{1}{|I||J|} \|P_{I\times J} \phi \|^2_2
   \le \|\phi\|_{\BMO_{rect}^d}^2.
\end{multline*}
Hence for $b = \sum_{I,J \in \D} h_{I,J} b_{I,J} \in \BMO^d(\R^2)$,
$\Omega \subseteq \R^2$ open,
\begin{eqnarray*}
 \|P_\Omega [S^{(1)}, [S^{(2)},{R_R}_\phi]] b \|_2^2
    &\le& \|[S^{(1)}, [S^{(2)},{R_R}_\phi]] P_{\tilde \Omega} b \|_2^2 \\
& \le& \sum_{I \times J \in \tilde \Omega}
      \|\phi\|^2_{\BMO_{rect}^d(\R^2)} |b_{I,J}|^2\\ &\le&  \|\phi\|_{\BMO^d(\R^2)_{rect} }^2 \|b\|_{\BMO^d(\R^2)}^2 |\tilde \Omega| \\
   &\le& 4  \|\phi\|_{\BMO^d([0,1]^2)}^2 \|b\|_{\BMO^d(\R^2)}^2 |\Omega|,
\end{eqnarray*}
where $\tilde \Omega= \cup_{I, J \in \D, I \times J \subseteq
\Omega} \tilde I \times \tilde J$ and $\tilde I$, $\tilde J$ are
the parents of $I$, $J$. Thus
\begin{equation}
    \|[S^{(1)}, [S^{(2)},{R_R}_\phi]]\|_{\BMO^d([0,1]^2) \rightarrow \BMO^d(\R^2)} \le 2 \|\phi\|_{\BMO_{rect}^d([0,1]^2)}.
\end{equation}

Here $\BMO_{rect}^d(\R^2)$ is the dyadic rectangular $\BMO$ space which continuously contains $\BMO^d(\R^2)$ and consists of function $f\in L^2_0(\R^2)$ such that
$$\sup_{I\times J\in \D(\R)^2}     \|  P_{I \times J} f\| =   \sup_{I\times J\in \D(\R)^2}\frac{1}{|I||J|}
\|f(s,t) - m_I f(t) - m_J f (s) + m_{I \times J}f \|_2^2<\infty.$$

For the boundedness of $ [S^{(1)},
  [S^{(2)},{\Pi_R}_\phi]]$ and  $ [S^{(1)},
  [S^{(2)},{\Delta_R}_\phi]]$ from $\BMO^d([0,1]^2)$ to $\BMO^d(\R^2)$, we remark that since
  $S^{(1)}$ is bounded on $\BMO^d(\R^2)$, we only need to show that
  $[S^{(2)},{\Delta_R}_\phi]$ and $[S^{(2)},{\Pi_R}_\phi]$ are
  bounded from $\BMO^d([0,1]^2)$ to $\BMO^d(\R^2)$.

  Straightforward computations give us for $b\in \BMO^d([0,1]^2)$,
\begin{eqnarray*}[S^{(2)},{\Delta_R}_\phi](b)(s,t)
&=&
  \sum_{I,J}b_{IJ}\phi_{IJ}h_I^2(s) \frac{h_{J^-}(t)-h_{J^+}(t)}{|J|^{1/2}}\\
&=&
 \frac{1}{2\sqrt 2}\sum_{I,J}  b_{IJ}\phi_{IJ}h_I^2(s) \left( \frac{\chi_{J^{-+}}(t)}{|J^{-+}|} -
 \frac{\chi_{J^{--}}(t)}{|J^{--}|} - \frac{\chi_{J^{++}}(t)}{|J^{++}|} +  \frac{\chi_{J^{+-}}(t)}{|J^{+-}|} \right) \\
&=& \frac{1}{2\sqrt 2}\Delta_{\tilde {\phi}}(\tilde {b})(s,t),
\end{eqnarray*}
  where
$\tilde b(s,t)=\sum_{I,J}b_{IJ}h_I(s)(h_{J^{-+}}(t) -h_{J^{--}}(t)
- h_{J^{++}}(t) +h_{J^{+-}}(t))$ and $\tilde
\phi=\sum_{I,J}\phi_{I,J}h_I(s)(h_{J^{-+}}(t) -h_{J^{--}}(t) -
h_{J^{++}}(t) +h_{J^{+-}}(t))$.

We obtain in the same way
 \begin{eqnarray*}[S^{(2)},{\Pi_R}_\phi](b)(s,t) &=&
  \sum_{I,J}\phi_{IJ} m_I(b_J) h_I(s)\frac{h_{J^-}(t)-h_{J^+}(t)}{|J|^{1/2}}\\
&=& \frac{1}{2\sqrt 2} \sum_{I,J}\phi_{IJ} m_I(b_J) h_I(s) \left(
\frac{\chi_{J^{-+}}(t)}{|J^{-+}|} -
 \frac{\chi_{J^{--}}(t)}{|J^{--}|} - \frac{\chi_{J^{++}}(t)}{|J^{++}|} +  \frac{\chi_{J^{+-}}(t)}{|J^{+-}|} \right) \\
 &=&  \frac{1}{2\sqrt 2}\Pi_{\Delta_{\tilde {\phi}}}(\tilde {b})(s).
\end{eqnarray*}
Since $\|\tilde b\|_{\BMO^d([0,1]^2)} \lesssim \|b\|_{\BMO^d([0,1]^2)}$ and
$\|\tilde \phi\|_{\LMO_{1}^d([0,1]^2)} \lesssim \|\phi\|_{\LMO_{1}^d([0,1]^2)}$,
we obtain
\begin{equation}
  \|[S^{(1)}, [S^{(2)},{\Delta_R}_\phi]]\|_{\BMO^d([0,1]^2) \to \BMO^d(\R^2)} \lesssim \|\phi\|_{\BMO^d([0,1]^2)}
\end{equation}
and
\begin{equation}
  \|[S^{(1)}, [S^{(2)},{\Pi_R}_\phi]]\|_{\BMO^d([0,1]^2) \to \BMO^d(\R^2)} \lesssim \|\phi\|_{\LMO_{1}^d([0,1]^2)}
\end{equation}
by  Theorem \ref{thm:others}. Swapping variables yields
\begin{equation}
  \|[S^{(1)}, [S^{(2)},{R_\Delta}_\phi]]\|_{\BMO^d([0,1]^2) \to \BMO^d(\R^2)} \lesssim \|\phi\|_{\BMO^d([0,1]^2)}
\end{equation}
and
\begin{equation}
  \|[S^{(1)}, [S^{(2)},R_{\Pi_\phi}]]\|_{\BMO^d([0,1]^2) \to \BMO^d(\R^2)} \lesssim \|\phi\|_{\LMO_{2}^d([0,1]^2)}.
\end{equation}
This finishes the proof of Theorem \ref{thm:dyshift}.
\end{proof}

To finish the proof of Theorem \ref{thm:hankel},  we need to
consider the relation between $\BMO(\R^N)$ and
$\BMO^d(\R^N)$  established in \cite{pipher,treil}.

Let us momentarily return to the one-variable setting and recall the following result which simplifies the one in \cite{pet}.

\begin{theorem}(Theorem 1.1 of \cite{hyt})  For $r\in [1, 2)$ and $\beta\in  \{0, 1\}^{\mathbb {Z}}$, let $S^{\beta,r}$ be the dyadic shift associated to the dyadic system $r\mathcal {D}^{\beta}$. Let $\mu$ stand for the canonical probability measure on $\{0, 1\}^{\mathbb {Z}}$ which makes the coordinate
functions $\beta_j$ independent with $\mu(\beta_j = 0) = \mu(\beta_j = 1) = 1/2$. Then for all $p\in  (1,\infty)$ and $f\in L^p(\R)$,
\begin{equation}
Hf (x)=-\frac{8}{\pi}\int_1^2\int_{\{0, 1\}^{\mathbb {Z}}}S^{\beta,r}f(x)d\mu(\beta)\frac{dr}{r},
\end{equation}

where the integral converges both pointwise for a.e. $x\in \R$ and also in the sense of an $L^p(\R)$-valued Bochner integral.
\end{theorem}

Returning to the two-variable setting, we obtain for $b \in \BMO([0,1]^2)$, $\phi \in
\LMO([0,1]^2)$:

$$
[H_1, [H_2, \phi]] b=\frac{64}{\pi^2}\int_1^2\int_1^2\int_{\{0, 1\}^{\mathbb {Z}}}\int_{\{0, 1\}^{\mathbb {Z}}}[S^{\alpha^1, r_1},[S^{\alpha^2, r_2}, \phi]] \, b \,d\mu(\alpha^1)\frac{dr_1}{r_1}\,d\mu(\alpha^2)\frac{dr_2}{r_2}.
$$

Now since $b
\in \BMO([0,1]^2)$, $\phi \in \LMO([0,1]^2)$, we have that $b \in
\BMO^{d,\vec {\alpha}, \vec {r}}(\R^2)$, and $\phi \in
\LMO^{d,\vec {\alpha}, \vec {r}}(\R^2)$ for each $\vec {\alpha}=(\alpha^1, \alpha^2)\in \{0, 1\}^{\mathbb {Z}}\times \{0, 1\}^{\mathbb {Z}}$ and $\vec {r}=(r_1, r_2)\in [1,2)^2$ with uniformly bounded norm (see
e.~.g.~\cite{ferglac}). Thus there exists a constant $C>0$ such
that
$$
     \|[S^{\alpha^1, r_1},[S^{\alpha^2, r_2}, \phi]] b\|_{\BMO^{d,\vec {\alpha}, \vec {r}}}
     \le C \|b\|_{\BMO} \|\phi\|_{\LMO} \text{ for all } (\alpha^1, \alpha^2, r_1, r_2)
$$
by Theorem \ref{thm:dyshift}. By \cite{treil}, Remark 0.5 (see also \cite{pipher}), it follows that
$$\frac{64}{\pi^2}\int_1^2\int_1^2\int_{\{0, 1\}^{\mathbb {Z}}}\int_{\{0, 1\}^{\mathbb {Z}}}[S^{\alpha^1, r_1},[S^{\alpha^2, r_2}, \phi]] \, b \,d\mu(\alpha^1)\frac{dr_1}{r_1}\,d\mu(\alpha^2)\frac{dr_2}{r_2}
  \in \BMO(\R^2)
$$
with norm controlled by $\|b\|_{\BMO} \|\phi\|_{\LMO}$. The proof is complete.

\end{proof}

\section{Acknowledgements} It is a pleasure to thank Aline Bonami for many very useful discussions and comments.
 Prof Bonami's visit at the University of Glasgow was made possible
through a grant of the Edinburgh Mathematical Society. The authors also gratefully acknowledge support of
the Fields Institute of Mathematical Sciences and the ``Thematic Program on New Trends in Harmonic Analysis'',
where part of this work was completed.

\bibliographystyle{plain}

\end{document}